\documentclass[a4 paper,11pt]{amsart}

\usepackage{amsmath} %
\usepackage{amssymb} %
\usepackage{amscd} %
\usepackage{amsthm} %
\usepackage{mathrsfs} %
\usepackage{euscript}
\usepackage{eufrak}
\usepackage{ulem}
\usepackage[alphabetic]{amsrefs}
\usepackage{ascmac} 
\usepackage{comment}
\usepackage{bm}

\usepackage[dvipdfmx]{graphicx} 

\usepackage{hyperref}
\usepackage{xcolor}
\definecolor{unbleu}{rgb}{0.03, 0.15, 0.4}

\hypersetup{
pdfborder = {0 0 0},
colorlinks,
linkcolor=unbleu,
citecolor=unbleu,
urlcolor=unbleu
}

\usepackage{fancyhdr}
 \newtheorem{theorem}{Theorem}[section]
 \newtheorem{lemma}[theorem]{Lemma}
 \newtheorem{proposition}[theorem]{Proposition}

\newtheorem{corollary}[theorem]{Corollary}

\newtheorem{maintheorem}{Theorem} 
  
\theoremstyle{definition}

\newtheorem{remark}[theorem]{Remark}

\def\ul#1{\underline{#1}}

\newcommand{\N}{\mathbb N}

\newcommand{\Int}{{\rm int}}

\newcommand{\Rot}{{\rm Rot}}
\newcommand{\rv}{{\rm rv}}

\begin{document}

\title[]{Constrained ergodic optimization for generic continuous functions}

\author[S.Motonaga]{Shoya Motonaga}
\address{Research Organization of Science and Technology, Ritsumeikan University, 1-1-1 Noji-higashi, Kusatsu, Shiga 525-8577, Japan}
\email{motonaga@fc.ritsumei.ac.jp}

\author[M. Shinoda]{Mao Shinoda}
\address{Department of Mathematics, Ochanomizu University, 2-1-1 Otsuka, Bunkyo-ku, Tokyo, 112-8610, Japan}
\email{shinoda.mao@ocha.ac.jp}


\begin{abstract}
One of the fundamental results of ergodic optimization asserts that
for any dynamical system on a compact metric space with the specification property and
for a generic continuous function $f$ every invariant probability measure that maximizes the space average of $f$ must have zero entropy. 
We establish the analogical result in the context of constrained ergodic optimization, which is introduced by Garibaldi and Lopes (2007).
\end{abstract}

\maketitle

\section{Introduction}

Let $T:X\rightarrow X$ be a continuous map on a compact metric space and $\mathcal{M}_T(X)$ be the space of Borel probability measures endowed with the weak*-topology. 
Let $C(X, \mathbb{R})$ be the space of real-valued continuous functions on $X$ with the supremum norm $\|\cdot\|_\infty$. For each $f\in C(X,\mathbb{R})$ we consider the maximum ergodic average
\begin{align}
    \beta(f)=\sup_{\mu\in \mathcal{M}_T(X)}\int f\ d\mu
    \label{beta}
\end{align}
and the set of all maximizing measures of $f$
\begin{align}
    \mathcal{M}_{{\rm max}}(f)=\left\{\mu\in \mathcal{M}_T(X): \int f\ d\mu=\beta(f)\right\}.
    \label{maxset}
\end{align}
The functional $\beta$ and the set $\mathcal{M}_{{\rm max}}(f)$ are main objects in ergodic optimization, which has been actively studied for a decade (see for more details \cite{Jen06survey, Jen2017}).

Constrained ergodic optimization, which is introduced by Garibaldi and Lopes in \cite{GarLop07}, investigates the analogical objects as \eqref{beta} and \eqref{maxset} under some constraint. 
Introducing a continuous $\mathbb{R}^d$-valued function $\varphi\in C(X,\mathbb{R}^d)$
playing the role of a constraint, 
we define the {\it rotation set} of $\varphi=(\varphi_1, \ldots, \varphi_d)$ by
\begin{align*}
    {\rm Rot}(\varphi)=\left\{\rv_{\varphi}(\mu)\in \mathbb{R}^d: \mu\in \mathcal{M}_T(X)\right\},
\end{align*}
where $\rv_{\varphi}(\mu)=(\int \varphi_1 d\mu, \ldots, \int\varphi_d d\mu)$.
For $h\in {\rm Rot}(\varphi)$, the fiber $\rv_{\varphi}^{-1}(h)$ is called the {\it rotation class} of $h$.
A point $h\in {\rm Rot}(\varphi)$ is called a {\it rotation vector}.
The terminology comes from Poincar\'e's rotation numbers for circle homeomorphisms. Many authors study the properties of rotation sets and characterize several dynamics in terms of rotation vectors \cite{GM99,Jen01rotation,K92, K95,KW14,KW16,KW19}.

We define the maximum ergodic average for a continuous function $f$ with constraint $h\in {\rm Rot}(\varphi)$ by
\begin{align}
    \beta^\varphi_h(f)=\sup_{\mu\in \rv_{\varphi}^{-1}(h)}\int f\ d\mu
    \label{beta_constraint}
\end{align}
and the set of all relative maximizing measures of $f$ with a rotation vector $h\in {\rm Rot}(\varphi)$ by
\begin{align*}
    \mathcal{M}^\varphi_h(f)=\left\{\mu\in \mathcal{M}_{T}(X): \beta^\varphi_h(f)=\int f\ d\mu,\ \rv_{\varphi}(\mu)=h\right\}.
\end{align*}
Note that this 
formulation is a generalization of (unconstrained) ergodic optimization, since for a constant constraint $\varphi\equiv h$ and its unique rotation vector $h$ the rotation set of $h$ is $\mathcal{M}_T(X)$.



We have much interest in the above constrained optimization because the constraints provide information about the asymptotic behavior of orbits.
Such problem is originally studied by Mather \cite{Mather91} and Ma\~{n}\'{e} \cite{Mane96} for Euler-Lagrange flows, where the asymptotic homological position of the trajectory in the configuration space is given as a constraint.
Recently, Bochi and Rams \cite{BochiRams16} proved that the Lyapunov optimizing measures for one-step cocycles of $2 \times 2$ matrices have zero entropy on the Mather sets under some conditions, which implies that a low complexity phenomena occurs 
in noncommutative setting such as a Lyapunov-optimization problem for one-step cocycles.
A relative Lyapunov-optimization is mentioned for their future research but to the authors’ knowledge the classical commutative counterpart, typicality of zero entropy of relative maximizing measure, has not been established yet.
We remark that Zhao \cite{Zhao19} studies constrained ergodic optimization for asymptotically additive potentials for the application in the study of multifractal analysis.



It is natural to extend the fundamental results of unconstrained ergodic optimization to constrained one.
In \cite{GarLop07}, several general results on constrained ergodic optimization are provided, especially, uniqueness of maximizing measures with any constraint for generic continuous functions is asserted.
Moreover, prevalent uniqueness of maximizing measures for continuous functions in the constrained settings easily follows from \cite{Mor21} (see Appendix \ref{AppendixB}).
However, in these studies, the differences between constrained ergodic optimization and unconstrained one are not mentioned explicitly.
In some cases, constraints prevent existence of relatively maximizing periodic measures (see Remark \ref{Rmk:periodic}), which should give rise to the problem of existence of periodic measures in a given rotation class.
Moreover, in contrast to Morris’ theorem \cite{Mor2010} which asserts that for any dynamical systems with the specification property every maximizing measure for a generic continuous function has zero entropy, we can easily verify that there exists a constraint such that for a generic potential its unique relatively maximizing measure has positive entropy (see Proposition~\ref{positive_entropy}).
Thus it is important to investigate the condition that the statement of Morris’s theorem holds for constrained ergodic optimization.

In this paper, we study the structure of rotation classes and the generic property in constrained ergodic optimization for symbolic dynamics. Our first main result is the density of periodic measures with some rational constraints. 
This is an analogical result to Sigmund's work for a dynamical system with the specification \cite{Sigmund}.
The difficulty in the constrained case comes from the existence of a measure which is a convex combination of ergodic measures with different rotation vectors.
This prevents us to use the ergodic decomposition in a given rotation class. 
To overcome this difficulty, our approach requires a certain finiteness for both of a subshift and a constraint function. 
Moreover, a detailed analysis is needed to construct a periodic measure approximating a given invariant one with the same rotation vector.
Let $(\Omega, \sigma)$ be an irreducible subshift of finite type 
with finite alphabets
(See \S\ref{symbolic} below).
A function $\varphi\in C(\Omega, \mathbb{R}^d)$ is said to be {\it locally constant} if for each $i=1, \ldots, d$ there exists $k_i\geq0$ such that
\begin{align*}
    \varphi_i(x)=\varphi_i(y)
    \quad \mbox{if}\quad
    x_j=y_j\ \mbox{for}\ j=0, \ldots, k_i.
\end{align*}
Denote by $\mathcal{M}_\sigma^p(\Omega)$ the set of invariant measures 
supported on a single periodic orbit.
\begin{maintheorem}
\label{density_of_periodic_measures}
Let $(\Omega, \sigma)$ be an irreducible subshift of finite type. 
Let $\varphi=(\varphi_1, \ldots, \varphi_d)\in C(\Omega, \mathbb{Q}^d)$ be a locally constant function and $h\in \Int(\Rot(\varphi))\cap \mathbb{Q}^d$.
Then the set
$\rv_{\varphi}^{-1}(h)\cap \mathcal{M}_\sigma^p(\Omega)$ is dense in $\rv_{\varphi}^{-1}(h)$.
\end{maintheorem}

\begin{remark}
Theorem \ref{density_of_periodic_measures} is motivated by Theorem 10 in \cite{GarLop07}: It was shown that
for a Walters potential $A$ on a subshift of finite type $\Omega$, there exists a periodic measure $\mu$ whose action $\int A \ d\mu$ approximates $\int A \ d\nu$ for $\nu\in \rv_{\varphi}^{-1}(h)$ with $\rv_{\varphi}(\mu)=h$ if $\nu$ is ergodic and the locally constant constraint $\varphi\in C(\Omega, \mathbb{Q}^d)$ is  joint recurrent in relation to $\nu$ (see \cite{GarLop07} for the definition of the joint recurrence and their precise statement).
Note that Theorem 10 in \cite{GarLop07} was suggested by the fact that a circle homeomorphism with rational rotation number has a periodic point and a result of this kind for symbolic dynamics was studied by Ziemian (see Theorem 4.2 of  \cite{Zie95}).

\end{remark}

\begin{remark}{
Although the  properties of rotation sets are well-studied in \cite{Zie95}, to the authors’ knowledge, that of rotation classes have attracted little attention.
We emphasize that Theorem \ref{density_of_periodic_measures} provides a more detailed description of Theorem 4.2 in  \cite{Zie95} under weaker assumptions. 
}
\end{remark}

\begin{remark}\label{Rmk:periodic}
If a constraint $\varphi\in C(\Omega,\mathbb{Q}^d)$ is locally constant, it is easy to see that the rotation vector of a periodic measure should be rational, i.e., 
\begin{align*}
    \varphi(\mathcal{M}_\sigma^p(\Omega))\subset {\rm Rot}(\varphi)\cap \mathbb{Q}^d.
\end{align*}
Hence in Theorem \ref{density_of_periodic_measures} we need to choose a rotation vector $h$ in ${\rm Rot}(\varphi)\cap \mathbb{Q}^d$.
\end{remark}

Applying Theorem \ref{density_of_periodic_measures}, we next investigate the property of relative maximizing measure for symbolic dynamics
.
Regarding \eqref{beta_constraint} as a functional on $C(\Omega,\mathbb{R})$, we can characterize a relative maximizing measure as a ``tangent" measure.
This characterization allows us to adapt argument in \cite{Mor2010} for our constrained case (See \S\ref{generic} for more details)
and we obtain the following.
\begin{maintheorem}
\label{generic_zero_entropy}
Let $(\Omega, \sigma)$ be an irreducible subshift of finite type. 
Let $\varphi=(\varphi_1, \ldots, \varphi_d)\in C(\Omega, \mathbb{Q}^d)$ be a locally constant function and $h\in \Int(\Rot(\varphi))\cap \mathbb{Q}^d$.
Then for generic $f\in C(\Omega,\mathbb{R})$ every relative maximizing measure of $f$ with constraint $h\in \Int(\Rot(\varphi))\cap \mathbb{Q}^d$ has zero entropy.
In particular, setting $K(f):=\bigcup_{\mu\in \mathcal{M}_h^\varphi(f)} {\rm supp}\mu $, we have $h_{{\rm top}}(K(f))=0$ for generic $f\in C(\Omega,\mathbb{R})$.
\end{maintheorem}

The remainder of this paper is organized as follows.
In \S\ref{density}, we study the structure of rotation classes. In particular, we clarify our definitions and notations for symbolic dynamics and prove Theorem \ref{density_of_periodic_measures} in \S\ref{symbolic}.
In \S\ref{generic} we will illustrate that a relative maximizing measure is regarded as a tangent measure of \eqref{beta_constraint} and prove Theorem \ref{generic_zero_entropy}.


\section{Structure of rotation classes}
\label{density}
\subsection{Density of convex combinations of periodic measures}
 	We first investigate the structure of rotation classes
 	for a continuous map $T:X\rightarrow X$ on a compact metric space $X$
 	such that $\mathcal{M}_T^p(X)$ is dense in $\mathcal{M}_T(X)$.
 	As mentioned in Remark \ref{Rmk:periodic},
 	a rotation class does not contain a periodic measure in some cases.
	Nevertheless, we have the density of convex combinations of periodic measures in a rotation class for every continuous constraint function and for every rotation vector in the interior of the rotation set.

	For $\nu\in \mathcal{M}_T(X)$, a finite set $F\subset C(X, \mathbb{R})$ and $\varepsilon>0$ set
\begin{align*}
    U_{(F, \varepsilon)}(\nu):=\left\{\mu\in \mathcal{M}_T(X): \left|\int f d\mu-\int f d\nu\right|<\varepsilon, f\in F\right\}.
\end{align*}
Let
\begin{align*}
    \Delta^d&=\left\{(\lambda_1,\ldots,\lambda_{d})\in (0,1)^{d}: \sum_{i=1}^{d}\lambda_i=1\right\},\\
    \mathcal{N}_T&=\Bigg\{\sum_{i=1}^{d+1} \lambda_i\mu_i:
        \ \mu_i\in \mathcal{M}_T^p(X),\quad
        (\lambda_1,\ldots,\lambda_{d+1})\in\Delta^{d+1},\\ 
       &\qquad\quad
       \dim {\rm span} \{\rv_{\varphi}(\mu_1)-\rv_{\varphi}(\mu_{d+1}), \ldots, \rv_{\varphi}(\mu_{d})-\rv_{\varphi}(\mu_{d+1})\}=d\Bigg\},
\end{align*}
 where ${\rm span}\{h_1, \ldots, h_{d}\}$ is the vector space spanned by $\{h_1, \ldots, h_d\}$.
	We begin with the following proposition.
 \begin{proposition}\label{density_of_rational_convex_combinations}
 Let $T:X\rightarrow X$ be a continuous map on a compact metric space $X$ such that $\mathcal{M}_T^p(X)$ is dense in $\mathcal{M}_T(X)$. 
 Let $\varphi=(\varphi_1, \ldots, \varphi_d)\in C(X, \mathbb{R}^d)$ and $h\in \Int(\Rot(\varphi))$.
Then the set
$\rv_{\varphi}^{-1}(h)\cap\mathcal{N}_T$ is dense in $\rv_{\varphi}^{-1}(h)$.
\end{proposition}

\begin{proof}
Take $\nu\in \rv_{\varphi}^{-1}(h)$ and a open neighborhood $U_\nu$ of $\nu$.
Then there exists a finite set $F\subset C(X, \mathbb{R})$ and $\varepsilon>0$ such that $U_{(F, \varepsilon)}(\nu)\subset U_\nu$.

Since $h\in\Int(\Rot(\varphi))$, there exists $\delta>0$ such that
\begin{align*}
    B_{\sqrt{d}\delta}(h)\subset \Rot(\varphi).
\end{align*}
where $B_{\sqrt{d}\delta}(h)$ is the open ball of radius $\sqrt{d}\delta$ centered at $h$.
Let $h'_i=h+\delta e_i$ for $i=1,\ldots, d+1$ where $\{e_i\}_{i=1}^d$ is the standard basis of $\mathbb{R}^d$ and $e_{d+1}=-\sum_{i=1}^d e_i$.
Then $\{h'_i\}_{i=1}^{d+1}$ forms a simplex with the barycenter $h$.
Let $\xi_i\in\rv_\varphi^{-1}(h'_i)$. 

Fix 
\begin{align*}
    0<t< \frac{\varepsilon}{4\max_{f\in F}\|f\|_\infty}
\end{align*}
and define $\nu_i:=(1-t)\nu+t\xi_i$.
Then for $f\in F$ we have
\begin{align*}
    \left|\int f d\nu_i-\int f d\nu\right|&\leq t\left|\int f d\xi_i-\int f d\nu\right|
    \leq 2t\|f\|_\infty<\frac{\varepsilon}{2}
\end{align*}
and $\nu_i\in U_{(F, \varepsilon/2)}(\nu)$.
By the definition of $\nu_i$ we have
\begin{align*}
    \rv_\varphi(\nu_i)=(1-t)\rv_\varphi(\nu)+t\rv_\varphi(\xi_i)
    =(1-t)h+th'_i=h+t\delta e_i.
\end{align*}
Thus $\frac{1}{d+1}\sum_{i=1}^{d+1}\rv_\varphi(\nu_i)=h$ holds.

Since $\mathcal{M}_T^p(X)$ is dense in $\mathcal{M}_T(X)$, for $r>0$, there exists 
\begin{align*}
    \mu_i\in U_{(F, \varepsilon/2)}(\nu_i) \cap U_{(\{\varphi\}, r/\sqrt{d})}(\nu_i) \cap \mathcal{M}_T^p(X).
\end{align*}
Since $\rv_{\varphi}(\nu_1)-\rv_{\varphi}(\nu_{d+1}),\rv_{\varphi}(\nu_2)-\rv_{\varphi}(\nu_{d+1}),\ldots, \rv_{\varphi}(\nu_d)-\rv_{\varphi}(\nu_{d+1})$ are linearly independent,
by the open property of linearly independence and the continuity of the map $\rv_\varphi$, we see that $\rv_{\varphi}(\mu_1)-\rv_{\varphi}(\mu_{d+1}),\rv_{\varphi}(\mu_2)-\rv_{\varphi}(\mu_{d+1}),\ldots, \rv_{\varphi}(\mu_d)-\rv_{\varphi}(\mu_{d+1})$ are also linearly independent for all sufficiently small $r>0$.
Moreover, we obtain
\begin{align*}
    \left|\left|\frac{1}{d+1}\sum_{i=1}^{d+1}\rv_\varphi(\mu_i)-h\right|\right|_{\mathbb{R}^d}=
    \left|\left|\frac{1}{d+1}\sum_{i=1}^{d+1}\rv_\varphi(\mu_i)-\frac{1}{d+1}\sum_{i=1}^{d+1}\rv_\varphi(\nu_i)\right|\right|_{\mathbb{R}^d}\\
    \le\frac{1}{d+1}\sum_{i=1}^{d+1}\left|\left|\rv_\varphi(\mu_i)-\rv_\varphi(\nu_i)\right|\right|_{\mathbb{R}^d}
    <\frac{d+1}{d+1}r=r,
\end{align*}
where $\|\cdot\|_{\mathbb{R}^d}$ is the standard norm of $\mathbb{R}^d$.
Hence $h$ is contained in the open ball of radius $r$ centered at $\frac{1}{d+1}\sum_{i=1}^{d+1}\rv_\varphi(\mu_i)$.
Taking $r>0$ sufficiently small, we deduce that $h$ is an interior point of the simplex whose vertices are $\rv_\varphi(\mu_1),\ldots, \rv_\varphi(\mu_{d+1})$,
which implies that there exists $(\lambda_1,\ldots,\lambda_{d+1})\in\Delta^{d+1}$ such that
$h=\sum_{i=1}^{d+1} \lambda_i \rv_\varphi(\mu_i)$.

Let $\tilde{\nu}=\sum_{i=1}^{d+1} \lambda_i\mu_i$. Trivially, $\rv_\varphi(\tilde{\nu})=h$ holds.
Then for every $f\in F$ we have
\begin{align*}
    \left|\int f d\tilde{\nu}-\int f d\nu\right|
    &=\left|\sum_{i=1}^{d+1}\lambda_i\int f d\mu_i-\sum_{i=1}^{d+1}\lambda_i\int f d\nu\right|\\
    &\leq \sum_{i=1}^{d+1}\lambda_i \left|\int f d\mu_i-\int f d\nu\right|\\
     &\leq \sum_{i=1}^{d+1}\lambda_i \left(\left|\int f d\mu_i-\int f d\nu_i\right|+\left|\int f d\nu_i-\int f d\nu\right|\right)\\
     &<\sum_{i=1}^{d+1}\lambda_i\left(\frac{\varepsilon}{2}+\frac{\varepsilon}{2}\right)
     =\varepsilon,
\end{align*}
i.e., $\tilde{\nu}=\sum_{i=1}^{d+1} \lambda_i\mu_i\in U_{(F,\varepsilon)}(\nu)$,
and this is precisely the assertion of the proposition.
\end{proof}

\begin{remark}
\label{decomposition}
In the proof of Proposition \ref{density_of_rational_convex_combinations},
we have $\mu_i\in U_{(F, \varepsilon)}(\nu)$ for every $i\in \{1, \ldots, d+1\}$
since we have $\nu_i\in U_{(F, \varepsilon/2)}(\nu)$ and $\mu_i\in U_{(F,\varepsilon/2)}(\nu_i)$.
We will use this fact in the proof of Theorem \ref{density_of_periodic_measures}.
\end{remark}

\begin{corollary}\label{residual_set_of_zero_entropy}
    Assume the hypotheses of Proposition \ref{density_of_rational_convex_combinations}. Let $H_\mu$ be the entropy map of $T$. Then the set
    \begin{align*}
        \mathcal{Z}=\{\mu\in \rv_{\varphi}^{-1}(h): H_\mu=0\}
    \end{align*} is a residual subset of $\rv_{\varphi}^{-1}(h)$.
\end{corollary}
\begin{proof}
    The argument is similar to \cite{DGS}[Proposition 22.16, p.223].
    By Proposition \ref{density_of_rational_convex_combinations}, $\mathcal{N}_T\cap \rv_{\varphi}^{-1}(h)$ is dense in $\rv_{\varphi}^{-1}(h)$
    and thus $\mathcal{Z}$ is also.
    Moreover, upper semi-continuity of the entropy map $\mu\mapsto H_\mu$ implies for every $n\geq1$
\begin{align*}
    \mathcal{Z}_n:=\left\{\mu\in \rv_{\varphi}^{-1} (h): 0\leq H_\mu<\frac{1}{n}\right\} \supset \mathcal{N}_T\cap \rv_{\varphi}^{-1}
(h)\end{align*}
is nonempty, open and dense in $\rv_{\varphi}^{-1}(h)$.
Hence $\mathcal{Z}=\bigcap_{n\geq 1}\mathcal{Z}_n$ is a residual set in $\rv_{\varphi}^{-1}(h)$.
\end{proof}



\begin{lemma}\label{lem:rational_coefficient}
    Assume the hypotheses of Proposition \ref{density_of_rational_convex_combinations}.
    Let $\mu_1,\ldots,\mu_{d+1}\in\mathcal{M}_T^p(X)$
    such that $\rv_\varphi(\mu_1), \ldots, \rv_\varphi(\mu_{d+1})\in \mathbb{Q}^d$ and
    \begin{align}\label{eqn:dimension}
        \dim {\rm span} \{\rv_{\varphi}(\mu_1)-\rv_{\varphi}(\mu_{d+1}), \ldots, \rv_{\varphi}(\mu_{d})-\rv_{\varphi}(\mu_{d+1})\}=d.
    \end{align}
    Let $(\lambda_1,\ldots,\lambda_{d+1})\in\Delta^{d+1}$
    and $\tilde{\nu}=\sum_{i=1}^{d+1}\lambda_i\mu_i\in\mathcal{M}_T(X)$.
    Then $\lambda_i$ is rational for each $i=1,\ldots, d+1$ if $\rv_\varphi(\tilde{\nu})$ belongs to $\mathbb{Q}^d$.
\end{lemma}
\begin{proof}
    Since $\tilde{\nu}=\sum_{i=1}^{d+1}\lambda_i\mu_i$ and $\sum_{i=1}^{d+1}\lambda_i=1$, we have
    \begin{align}\label{eqn:rv}
        \rv_\varphi(\tilde{\nu})=\rv_\varphi(\mu_{d+1})+\sum_{i=1}^d \lambda_i\left(\rv_\varphi(\mu_i)-\rv_\varphi(\mu_{d+1})\right).
    \end{align}
    Let $V$ be a $d\times d$-matrix given by
\begin{align}\label{eqn:V}
    V=\Big(\rv_\varphi(\mu_{d+1})-\rv_\varphi(\mu_{1}), \ldots, \rv_\varphi(\mu_{d+1})-\rv_\varphi(\mu_{d})\Big)
\end{align}
and $\lambda$ be a column vector given by $(\lambda)_i=\lambda_i$ for $i=1,\ldots, d$.
It follows from \eqref{eqn:dimension} that the matrix $V$ is invertible.
Therefore, by \eqref{eqn:rv}, we obtain
\begin{align}\label{eqn:lambda}
    \lambda=V^{-1}\left(\rv_\varphi(\mu_{d+1})-\rv_\varphi(\tilde{\nu})\right).
\end{align}
Since $\rv_\varphi(\mu_1), \ldots, \rv_\varphi(\mu_{d+1})$ and $\rv_\varphi(\tilde{\nu})$ are in $\mathbb{Q}^d$,
 each component of $V$ and that of its inverse $V^{-1}$ are
rational.
Therefore, by \eqref{eqn:lambda}, we deduce that
$\lambda_i$ is rational for each $i=1,\ldots, d+1$.
\end{proof}

\subsection{Symbolic dynamics and density of periodic measures}
\label{symbolic}
We next consider symbolic dynamics. In this particular case, under some assumptions, we can prove the density of periodic measures in a given rotation class.
Denote by $\mathbb{N}_0$ the set of all non-negative integers.
For a finite set $\mathcal{A}$ we consider the one-sided infinite product $\mathcal{A}^{\mathbb{N}_0}$ equipped with the product topology of the discrete one.

Let $\sigma$ be the shift map on $\mathcal{A}^{\mathbb{N}_0}$ 
 (i.e.,~$(\sigma (\ul{x}))_i= x_{i+1}$ for each $i\in {\mathbb{N}_0}$ 
  and $\ul{x}= ( x_i)_{i\in {\mathbb{N}_0}} \in \mathcal{A}^{\mathbb{N}_0}$). 
When a subset  $\Omega$ of $\mathcal{A}^{\mathbb{N}_0}$ is $\sigma$-invariant and  closed,  
we call it a \textit{subshift}.
Slightly abusing the notation we denote by $\sigma$ the shift map restricted on $\Omega$.

For a subshift $\Omega$, let 
$[u] = \left\{ \ul{x}\in \Omega : u=x_0\cdots x_{n-1}\right\}$ for each $u\in \mathcal{A}^n$, $n\geq 1$
and set
$
\mathcal L(\Omega ) =\left\{ u\in \bigcup _{n\geq 1}\mathcal{A}^n : [u] \neq \emptyset \right\}
$.
We also denote $\mathcal{L}_n(\Omega):=\{u\in\mathcal{L}(\Omega):|u|=n\}$ for $n\ge 1$,
where $|u|$ denotes the length of $u$, i.e., $|u|=n$ if 
$u =u_0\cdots u_{n-1} \in \mathcal{A}^n$.
A word $u\in \mathcal{A}^n$ {\it appears in} $\ul{x}\in \mathcal{A}^{\mathbb{N}_0}$ if there exists $k\geq0$ such that $x_k\cdots x_{k+n-1}=u_{0}\cdots u_{n-1}$.
For $u,v\in\mathcal{L}(\Omega)$, we use the juxtaposition $uv$ to denote the word obtained by the concatenation and $u^\infty$ means a one-sided infinite sequence $uuu\cdots \in \mathcal{A}^{\mathbb{N}_0}$.
We say that $\Omega$ is \textit{irreducible} if for any $i,j\in A$,
we can find $u\in\mathcal{L}(\Omega)$ such that $iuj\in\mathcal{L}(\Omega)$ holds.
A subshift $\Omega$ is a {\it subshift of finite type (SFT)} if there exists a finite set $\mathcal{F}\subset \bigcup_{n\geq 1}\mathcal{A}^n$ such that no word from $\mathcal{F}$ appears in any $\underline{x}\in \Omega$.
The set $\mathcal{F}$ is called a forbidden set of $\Omega$.
Note that different forbidden sets may define the same subshift of finite type.

In order to prove Theorem \ref{density_of_periodic_measures} we shall show that on a subsfhit of finite type periodic orbits which share the same word can be concatenated without extra gap words.
\begin{lemma}
\label{concatenation}
Let $(\Omega,\sigma)$ be a subshift of finite type with a forbidden set $\mathcal{F}$. Let $\kappa=\max\{|u|: u \in \mathcal{F}\}$.
Let $u\in \mathcal{L}_{\kappa}(\Omega)$ and $v,w\in \mathcal{L}(\Omega)$ such that $(vu)^\infty, (wu)^\infty\in \Omega$.
Then for every $k\geq1$ and sequences  $\{m_i\}_{i=1}^k, \{n_i\}_{i=1}^k\subset \mathbb{N}$, we have
\begin{align*}
    ((vu)^{m_1}(wu)^{n_1}(vu)^{m_2}\cdots (vu)^{m_k}(wu)^{n_k})^\infty \in \Omega.
\end{align*}
\end{lemma}
\begin{proof}
Let $k\geq1$, $\{m_i\}_{i=1}^k$ and $\{n_i\}_{i=1}^k\subset \mathbb{N}$.
By the definition of $\kappa$ we should only check the words in $((vu)^{m_1}(wu)^{n_1}(vu)^{m_2}\cdots (vu)^{m_k}(wu)^{n_k})^\infty$ with length less than $\kappa$.
However such a word is a subword of 
$vu, uv, wu$ and $uw$.
Since $(vu)^\infty, (wu)^\infty \in \Omega$, there is no forbidden word in $vu, uv, wu$ and $uw$, which complete the proof.
\end{proof}

In addition, an irreducible sofic shift satisfies the specification property and thus $\mathcal{M}_\sigma^p(\Omega)$ is dense in $\mathcal{M}_\sigma (\Omega)$ (see for example \cite{Weiss73} and \cite{Sigmund}).
Hence we have the following by Proposition \ref{density_of_rational_convex_combinations}.

\begin{corollary}\label{cor:density_of_rational_convex_combinations}
Let $(\Omega, \sigma)$ be an irreducible subshift of finite type with a forbidden set $\mathcal{F}$.
Let $\varphi=(\varphi_1, \ldots, \varphi_d)\in C(\Omega, \mathbb{R}^d)$ be a 
locally constant function and $h\in \Int(\Rot(\varphi))$. 
Then the set $\rv_\varphi^{-1}(h)\cap\mathcal{N}_\sigma$ is dense in $\rv_\varphi^{-1}(h)$.
\end{corollary}

Using Lemma \ref{concatenation} and Corollary \ref{cor:density_of_rational_convex_combinations}, we can prove Theorem \ref{density_of_periodic_measures}.
\begin{proof}[Proof of Theorem \ref{density_of_periodic_measures}]
    Take $\nu\in \rv_{\varphi}^{-1}(h)$ and a open neighborhood $U_\nu$ of $\nu$.
There exists a finite set $F\subset C(\Omega, \mathbb{R})$ and $\varepsilon>0$ such that $U_{(F, \varepsilon)}(\nu)\subset U_\nu$.
Without loss of generality we may assume every $f\in F$ is locally constant.
Let $\kappa=\max\{|u|: u \in \mathcal{F}\}$
and $u=u_0\cdots u_{\kappa-1}\in \mathcal{L}_\kappa(\Omega)$ such that $\nu([u])>0$.
Replace $F$ and $\varepsilon$ with $F\cup\{\chi_{[u]}\}$ and $\min\{\varepsilon, \nu([u])/2\}$ respectively,
where $\chi_{[u]}$ is the characteristic function of $[u]$.

By Corollary \ref{cor:density_of_rational_convex_combinations},
there is $\tilde{\nu}\in U_{(F,\varepsilon/2)}(\nu)\cap\rv_\varphi^{-1}(h)$ of the form $\tilde{\nu}=\sum_{i=1}^{d+1} \lambda_i\mu_i$
where $\mu_i\in\mathcal{M}_\sigma^p(\Omega)$ and $(\lambda_1,\ldots,\lambda_{d+1})\in\Delta^{d+1}$ with \eqref{eqn:dimension}.
Note that $\rv_{\varphi}(\mu_i)\in \mathbb{Q}^d\ (i=1,\ldots,d+1)$ as stated in Remark \ref{Rmk:periodic}.
For each $i=1,\ldots,d+1$, we will denote by $a_i^\infty$ the corresponding periodic orbits to $\mu_i\in\mathcal{M}_\sigma^p(\Omega)$ where $a_i$ is the word with lengths of periods.
Moreover, by Lemma \ref{lem:rational_coefficient}, each $\lambda_i$ can be written as $\lambda_i=q_i/Q$ where $q_1,\ldots, q_{d+1},Q \in \mathbb{N}$
with $q_1+\ldots+q_{d+1}=Q$ since $(\lambda_1,\ldots,\lambda_{d+1})\in\Delta^{d+1}\cap\mathbb{Q}^{d+1}$.

Let $\ell$ be the maximum length of the words on which elements of $F\cup\{\varphi_1, \ldots, \varphi_d\}$ depend.
We can assume
\begin{align*}
   \ell<\min\{|a_1|,\ldots, |a_{d+1}|\}
\end{align*}
by replacing $a_i\ (i=1,\ldots,d+1)$ with concatenations $a_i^k$ for some $k\in \mathbb{N}$ if necessary.
As stated in Remark \ref{decomposition}, we have $\mu_i\in U_{(F,\varepsilon)}(\nu)$ for every $i=1, \ldots, d+1$, 
which implies $\mu_i([u])>\nu([u])-\varepsilon\geq \nu([u])/2>0$.
Hence $u$ is a subword of each $a_i$
 and without loss of generality we may assume $a_{i, 0}\cdots a_{i, |u| -1}=u$.
For $g\in F\cup\{\varphi\}$ define $\delta_{g,k}\ (k=1,\ldots,d+1)$ by
\begin{align*}
    \delta_{g,k}=
    \sum_{i=0}^{\ell-1}
    \Big\{&g(\sigma^{|a_k|-\ell+i}(a_k a_{k+1}))
    -g(\sigma^{|a_k|-\ell+i}(a_k a_k))
    \Big\}
\end{align*}
where $a_{d+2}=a_1$.
Set $\delta_{g}=\delta_{g,1}+\ldots+\delta_{g,d+1}$.

Let $V$ be a $d\times d$-matrix given by \eqref{eqn:V}.
As stated in the proof of Lemma \ref{lem:rational_coefficient}, the matrix $V$ is invertible by \eqref{eqn:dimension}.
Since each component of $V$ and $\delta_{\varphi}$ is rational,
we denote $V^{-1}\delta_{\varphi}=v/R$ where $v\in \mathbb{Z}^d$ and $R\in \mathbb{N}$. Let $v_i=(v)_i\ (i=1,\ldots, d)$.

Now we construct a periodic measure near $\nu$ with the rotation vector $h$.
Since $a_1, \ldots, a_{d+1}$ share the same word $u$, by Lemma \ref{concatenation}, we can concatenate them without extra gap words.
Hence let
\begin{align*}
    A&=|a_1|\ldots|a_{d+1}|,\quad 
    C=\max_{f\in F,\ i=1,\ldots,d+1} \left|\int f \ d \mu_i\right|,\quad \delta^*=\max_{f\in F} |\delta_f|,\\
    m_i'&=Av_i/|a_i|\ (i=1,\ldots,d),\quad
    m_{d+1}'=-A(\sum_{i=1}^d v_i)/|a_{d+1}|,\\
    M_j'&=Aq_j/|a_j|,\quad
    M_j=tARM_j'+m_j'\  (j=1,\ldots, d+1),\\
    y&=a_1^{tM_1'+m_1'}a_2^{tM_2'+m_2'}\cdots a_{d+1}^{tM_{d+1}'+m_{d+1}'},\\
    z&=a_1^{tM_1'}a_2^{tM_2'}\cdots a_{d+1}^{tM_{d+1}'},\quad
    x=yz^{(AR-1)},
\end{align*}
where $t\in \N$ is large enough to satisfy
\begin{align}\label{ineq_t}
    \frac{AR}{|x|}\delta^*+C\sum_{i=1}^{d+1}\left|\frac{M_i|a_i|}{|x|}-\lambda_i\right|<\frac{\varepsilon}{2}
\end{align}
and $tM_{d+1}'-A(\sum_{i=1}^d v_i)/|a_{d+1}|>0$.
Note that such $t\in\N$ exists since
\begin{align*}
    |x|&=\sum_{i=1}^{d+1} M_i|a_i|=tAR\sum_{i=1}^{d+1} M_i'|a_i|,\\
    \lambda_i&=\frac{q_i}{\sum_{j=1}^{d+1}q_j}=\frac{M_i'|a_i|}{\sum_{j=1}^{d+1} M_j'|a_j|}
\end{align*}
hold and the left hand side of \eqref{ineq_t} tends to $0$ as $t\to +\infty$.
Let $\mu$ be the periodic measure supported on $x^\infty$.

First we check $\rv_{\varphi} (\mu)=h (=\rv_{\varphi}(\tilde{\nu}))$.
Since $\varphi$ is locally constant,
\begin{align*}
    \rv_{\varphi}(\mu)&=\frac{1}{|x|}S_{|x|}\varphi(x)\\
    &=\frac{1}{|x|}\left(AR\delta_{\varphi}+\sum_{i=1}^{d+1} M_i|a_i|\int \varphi d\mu_i\right)\\
    &=\frac{1}{|x|}AR\delta_{\varphi}+\frac{1}{|x|}\sum_{i=1}^{d+1} m_i'|a_i|\rv_{\varphi}(\mu_i)+\sum_{i=1}^{d+1}\lambda_i\rv_{\varphi}(\mu_i)\\
    &=h+\frac{A}{|x|}Vv+\frac{1}{|x|}\sum_{i=1}^{d+1} m_i'|a_i|\rv_{\varphi}(\mu_i)\\
    &=h+\frac{A}{|x|}\left\{Vv+\sum_{i=1}^{d}v_i\rv_{\varphi}(\mu_i)-(\sum_{j=1}^d v_j)\rv_{\varphi}(\mu_{d+1})\right\}
    =h.
\end{align*}

Next we check $\mu\in U_{(F, \varepsilon)}(\nu)$.
Let $f\in F$. We compute
\begin{align*}
    &\left|\frac{1}{|x|}S_{|x|}f-\int f\ d\nu\right|\\
    &\leq \left|\frac{1}{|x|}S_{|x|}f-\int f\             d\tilde{\nu}\right|+\left|\int f\ d\tilde{\nu}-\int f\ d\nu\right|\\
    &=\left|\frac{1}{|x|}\Big(AR\delta_{f}+\sum_{i=1}^{d+1} M_i|a_i|\int f d\mu_i\Big)-\sum_{j=1}^{d+1} \lambda_i\int f d\mu_i\right|+\frac{\varepsilon}{2}\\
    &\leq \sum_{i=1}^{d+1}\left|\left(\frac{M_i|a_i|}{|x|}-\lambda_i\right)\int f d\mu_i\right|
    +\frac{AR}{|x|}|\delta_f|+\frac{\varepsilon}{2}\\
    &< \frac{\varepsilon}{2}+\frac{\varepsilon}{2}= \varepsilon,
\end{align*}
which completes the proof.
\end{proof}

\begin{remark}
In the proof of Theorem \ref{density_of_periodic_measures},
we think the word $y$ as a corrective one to attain the desired rotation vector.
A similar approach is used in Theorem 5 of \cite{Jen01rotation} in a different setting
but our construction is more explicit than it.
\end{remark}

\begin{remark}
Note that the error term $\delta_\varphi$ does not depend on $m'_i\ (i=1,\ldots, d+1)$ in our case. For a subshift with the specification condition, we can concatenate the words $a_1^{m'_1}, \ldots,a_{d+1}^{m'_{d+1}}$ with some gap words but the error term in such case depends on $m'_i\ (i=1,\ldots, d+1)$,
which implies we cannot choose a suitable corrective word $y$ for the error term $\delta_\varphi$ in such case.
So we have to overcome this difficulty to extend Theorem \ref{density_of_periodic_measures} to the case of a subshift with the specification condition.
\end{remark}

\begin{remark}
For a rotation vector in the boundary of a rotation set, there may exist no periodic measure in the rotation class.
Let $\Omega\subset\{1,2,3\}^{\mathbb{N}_0}$ be a Markov shift with an adjacency matrix 
\begin{align*}
    A=\begin{pmatrix}1&1&0\\1& 1 &1\\0&1&1\end{pmatrix}
\end{align*}
(i.e., $\Omega=\{\ul{x}\in \{1,2,3\}^{\mathbb{N}_0}: A_{x_i x_{i+1}}=1\ \mbox{for all} \ i\in \mathbb{N}_0\}$) and define $\varphi=(\varphi_1,\varphi_2,\varphi_3):\Omega \rightarrow\mathbb{Q}^3$ by
\begin{align*}
    \varphi_i(\ul{x})=\left\{\begin{array}{cc}
        1 & x_0=i \\
        0 & \mbox{else.}
    \end{array}
    \right.
\end{align*}
Then its rotation set $\Rot(\varphi)$ is the polyhedron whose extremal points are 
$e_1, e_2$ and $e_3$
, where $\{e_1,e_2,e_3\}$ is the standard basis of $\mathbb{R}^3$.
Take a rotation vector $h$ from the open side whose vertices are $e_1$ and $e_3$, i.e., $h\in \{t e_1+(1-t)e_3: t\in (0,1)\}$.
If there exists a periodic measure $\mu\in \rv_\varphi^{-1}(h)$, the corresponding periodic orbit $u^\infty$ should contain both of $1$ and $3$.
Since there is no sequence including $13$ and $31$ in $\Omega$, the word $u$ must contain the symbol $2$.
Hence we have $\rv_{\varphi_2}(\mu)>0$ and $h=\rv_\varphi(\mu)\notin \{t e_1+(1-t)e_3: t\in (0,1)\}$, which is a contradiction.
\end{remark}

\section{Generic property for constraint ergodic optimization}
\label{generic}
In this section, we prove Theorem \ref{generic_zero_entropy}. Our proof is based on the approach presented by Morris \cite{Mor2010} for the unconstrained case but we need to pay careful attention to the constraint. Moreover, density of periodic measures in the rotation class (i.e., Theorem \ref{density_of_periodic_measures}) plays an important role to obtain the argument.

\subsection{Characterization by tangency}
We turn to a general dynamical system in this subsection.
Let $T: X\rightarrow X$ be a continuous map on a compact metric space. 
Denote by $\mathcal{M}_T^e(X)$ the set of all ergodic measures on $X$.
By the Riesz representation theorem a Borel probability measure on $X$ can be regarded as a bounded linear functional on $C(X,\mathbb{R})$.
Hence we use the operator norm $\|\mu\|=\sup\{|\mu(f)|: f\in C(X,\mathbb{R})\ \mbox{with}\ \|f\|_\infty=1\}$ for an invariant measure $\mu\in \mathcal{M}_T(X)$.

First we characterize a relative maximizing measures by tangency to \eqref{beta_constraint}.
Let $\varphi\in C(X,\mathbb{R}^d)$ and $h\in {\rm Rot}(\varphi)$. Note that here we do not need to assume that $h\in \Int(\Rot(\varphi))\cap \mathbb{Q}^d$.

\begin{lemma} 
\label{tangent}
$\mu\in \mathcal{M}^\varphi_h (f)$ iff $\mu$ is tangent to $\beta^\varphi_h$ at $f$.
\end{lemma}

\begin{proof}
Let $\mu\in \mathcal{M}^\varphi_h(f)$.
Then for every $g\in C(X,\mathbb{R})$ we have
\[
\beta^\varphi_h(f+g)-\beta^\varphi_h(f)\geq \int f+g\ d\mu-\int f\ d\mu \geq \int g\ d\mu.
\]

Let $\mu$ be tangent to $\beta^\varphi_h$ at $f$. 
For every $g\in C(X,\mathbb{R})$ we have
\begin{align*}
     \int g\ d\mu &\leq \beta^\varphi_h(f+g)-\beta^\varphi_h(f) \nonumber\\
        &=\beta^\varphi_h(g)+\beta^\varphi_h(f+g)-\beta^\varphi_h(f)-\beta^\varphi_h(g) \nonumber\\
        &\leq \beta^\varphi_h(g). \label{bounded}
\end{align*}
Then we can show that $\mu$ is an invariant probability measure in the same way as unconstrained case (See for example \cite{Shi2018, Bre08}).
We now see that $\mu$ takes the rotation vector $h$.
For $i=1, \ldots, d$ we have
 \begin{align*}
        \int \varphi_i\ d\mu\leq \beta^\varphi_h(\varphi_i)=h_i,
         \quad
         -\int \varphi_i\ d\mu\leq \beta^\varphi_h(-\varphi_i)=-h_i,
    \end{align*}
which yields $\rv_{\varphi}(\mu)=h$.

Finally we check $\int f\ d\mu=\beta^\varphi_h(f)$.
Indeed, 
\begin{align*}
    \int -f\ d\mu \leq \beta^\varphi_h(f-f)-\beta^\varphi_h(f)=-\beta_h^\varphi(f).
\end{align*}
Multiplying $-1$, we have
\begin{align*}
     \int f d\mu\geq \beta^\varphi_h(f).
\end{align*}
\end{proof}

Now Lemma \ref{tangent} allows us to adopt the techniques for unconstrained ergodic optimization.
While the rest of this subsection is similar to \cite{Mor2010}, attention should be paid to the constraint and thus we give proofs of Lemmas \ref{approx}, \ref{dense1}, \ref{dense2} below.
\begin{lemma}
\label{approx}
Let $f\in C(X,\mathbb{R})$ and $\varepsilon>0$. 
If $\nu\in \rv_{\varphi}^{-1}(h)\cap \mathcal{M}_{T}^e(X)$ such that $\beta^\varphi_h(f)-\int f\ d\nu<\varepsilon$, 
then there exists $g\in C(X,\mathbb{R})$ such that $\|f-g\|_\infty <\varepsilon$ and $\nu\in \mathcal{M}^\varphi_h(g)$.
\end{lemma}

\begin{proof}
Applying Bishop-Phelps's theorem \cite[Theorem V.1.1.]{Israel} to $f$, $\nu$ and $\varepsilon'=1$ we have $g\in C(X,\mathbb{R})$, $\eta\in \rv_{\varphi}^{-1}(h)$ such that  $\eta$ is tangent to $\beta^\varphi_h$ at $g$, $\|\nu-\eta\|<\varepsilon'=1$ and 
\[
\|f-g\|_\infty<\frac{1}{\varepsilon'}\left(\beta^\varphi_h(f)-\int f d\nu\right)<\varepsilon.
\]
By Lemma \ref{tangent} we have $\eta\in \mathcal{M}^\varphi_h(g)$.
If $\eta=\nu$, the proof is complete. 
Let $\eta \neq \nu$.
We can conclude $\nu$ and $\eta$ are not mutually singular by $\|\nu-\eta\|<1$ in the same way as in \cite{Mor2010}[Lemma 2.2]. 
Hence by the Lebesgue decomposition theorem, there exist $\hat{\nu}, \hat{\eta}\in \mathcal{M}$ and $\lambda\in (0,1)$ such that
\[
    \eta=(1-\lambda)\hat{\eta}+\lambda\hat{\nu}
\]
where $\hat{\eta}\perp \nu$ and $\hat{\nu}\ll\nu$.
By a standard argument using the Radon-Nikodym theorem, it is easy to see $\hat{\nu}=\nu$ (See for example \cite{Walters}).
Then for each $i=1, \ldots, d$
\begin{align*}
    h_i=\int \varphi_i\ d\eta&=(1-\lambda)\int \varphi_i\ d\hat{\eta}+\lambda \int \varphi_i\ d\nu\\
        &=(1-\lambda)\int \varphi_i\ d\hat{\eta}+\lambda h_i.
\end{align*}
Hence we have
\begin{align*}
    (1-\lambda)\left(\int \varphi_i\ d\hat{\eta}-h_i\right)=0
     \quad (i=1, \ldots, m),
\end{align*}
which yields $\hat{\eta}\in \rv_{\varphi}^{-1}(h)$.
Since $\hat{\eta}, \nu\in \rv_{\varphi}^{-1}(h)$, we have
\[
    \int g\ d\hat{\eta}\leq \beta^\varphi_h(g) \quad\mbox{and}\quad \int g\ d\nu\leq \beta^\varphi_h(g).
\]
They should be equality since $\int g\ d\eta =\beta^\varphi_h(g)$. Therefore, we obtain $\nu\in \mathcal{M}^\varphi_h(g)$.
\end{proof}

Set $\mathcal{E}=\rv_{\varphi}^{-1}(h)\cap \mathcal{M}^e_{T}(X)$.
In the reminder of this subsection we assume 
\begin{align}
   \overline{\mathcal{E}}=\rv_{\varphi}^{-1}(h)
   \label{e_dense}
\end{align}

Since the constraint requires no further discussion
in the next two lemmas, we omit the proofs.
\begin{lemma}
\label{open1}
Let $\mathcal{U}\subset \rv_{\varphi}^{-1}(h)$ be an open set.
Then 
\[
U:=\{f\in C(X,\mathbb{R}): \mathcal{M}^\varphi_h(f)\subset \mathcal{U}\}
\]
is open in $C(X,\mathbb{R})$.
\end{lemma}
%


\begin{lemma}
\label{open2}
Let $U\subset C(X,\mathbb{R})$ be an open set. Then
\[
\mathcal{U}:=\mathcal{E}\cap \bigcup_{f\in U} \mathcal{M}^\varphi_h(f)
\]
is open in $\mathcal{E}$.
\end{lemma}

At the end of this subsection, we give the following two lemmas.
\begin{lemma}
\label{dense1}
Let $\mathcal{U}$ be a dense subset of $\mathcal{E}$. 
Then 
\[
U:=\{f\in C(X,\mathbb{R}): \mathcal{M}^\varphi_h(f)=\{\mu\}\ \mbox{for some}\ \mu\in \mathcal{U}\}
\]
is dense in $C(X,\mathbb{R})$.
\end{lemma}

\begin{proof}
Let $V\subset C(X,\mathbb{R})$ be an open set. Set
\[
\mathcal{V}=\mathcal{E}\cap \bigcup_{f\in V}\mathcal{M}^\varphi_h(f).
\]
By Lemma \ref{open2}, $\mathcal{V}$ is open in $\mathcal{E}$.
Since $\mathcal{U}$ is dense in  $\mathcal{E}$, $\mathcal{U}\cap \mathcal{V}\neq \emptyset$.
Take $\mu\in \mathcal{U}\cap \mathcal{V}$ and let $f\in V$ such that $\mu\in \mathcal{M}^\varphi_h(f)$.
Note that $\mu\in\rv_\varphi^{-1}(h)$.
Since $\mu\in\mathcal{E}\subset \mathcal{M}_{T}^e(X)$, by Jenkinson's theorem \cite{Jen06}[Theorem 1], there exists $g\in C(X,\mathbb{R})$ such that $\mathcal{M}_{\rm max}(g)=\{\mu\}$. This implies $\mathcal{M}^\varphi_h(g)=\{\mu\}$ since $\mu\in\rv_{\varphi}^{-1}(h)$.
Hence for $\delta>0$ we have
$\mathcal{M}^\varphi_h(f+\delta g)=\{\mu\}$ and $f+\delta g\in U$.
For sufficiently small $\delta$ we have $f+\delta g\in V$, which complete the proof.
\end{proof}

\begin{lemma}
\label{dense2}
Let $U\subset C(X,\mathbb{R})$ be a dense subset. Then
\[
    \mathcal{U}:=\mathcal{E}\cap \bigcup_{f\in U} \mathcal{M}^\varphi_h(f)
\]
is dense in $\mathcal{E}$.
\end{lemma}
\begin{proof}
 Take a nonempty open subset $\mathcal{V}\subset \overline{\mathcal{E}}$. 
 We show $\mathcal{V}\cap\mathcal{U}\neq\emptyset$.
 
 By Lemma \ref{open1}, 
 \[
    V:=\{f\in C(X,\mathbb{R}): \mathcal{M}^\varphi_h(f)\subset \mathcal{V}\}
 \]
 is open in $C(X,\mathbb{R})$.
 Since $\mathcal{V}$ is nonempty, there exists $\mu\in \mathcal{V}\cap\mathcal{E}\subset \rv_\varphi^{-1}(h)\cap\mathcal{M}_{T}^e(X)$.
 Since $\mu$ is ergodic, as stated in Lemma \ref{dense1}, there exists $g\in C(X,\mathbb{R})$ such that $\mathcal{M}_{\rm max}(g)=\{\mu\}$ and we have $\mathcal{M}^\varphi_h(g)=\{\mu\}\subset \mathcal{V}$ by Jenkinson's Theorem \cite{Jen06}[Theorem 1].
 Hence $V$ is nonempty. 
 Since $U$ is dense in $C(X,\mathbb{R})$, $U\cap V\neq \emptyset$ and $\mathcal{U}\cap \mathcal{V}\neq \emptyset$.
\end{proof}

\subsection{Generic zero entropy for a symbolic dynamics}
In this subsection, we apply the lemmas in the previous subsection to the symbolic case and give the proof of Theorem \ref{generic_zero_entropy}.


%


\begin{proof}[Proof of Theorem \ref{generic_zero_entropy}]

Suppose the hypotheses of Theorem \ref{generic_zero_entropy} hold.
As in Corollary \ref{residual_set_of_zero_entropy}, for each $n\geq1$
\begin{align*}
    \mathcal{Z}_n:=\left\{\mu\in \rv_{\varphi}^{-1} (h): 0\leq H_\mu<\frac{1}{n}\right\} \supset \mathcal{N}_{\sigma}\cap \rv_{\varphi}^{-1}
(h)\end{align*}
 is nonempty, open and dense subset in $\rv_{\varphi}^{-1}(h)$.
Moreover, by Theorem \ref{density_of_periodic_measures}, we have
\begin{align*}
    \overline{\rv_{\varphi}^{-1}(h)\cap \mathcal{M}_\sigma^p(\Omega)}=\rv_{\varphi}^{-1}(h), 
\end{align*}
which implies \eqref{e_dense}.

Fix $n\geq1$. 
Since \eqref{e_dense} holds, we can apply Lemmas \ref{open1} and \ref{dense1} to our symbolic case. Therefore, we see that
\begin{align*}
    U_n:=\{f\in C(\Omega,\mathbb{R}): \mathcal{M}^\varphi_h(f)\subset \mathcal{Z}_n\}
\end{align*}
is open in $C(\Omega,\mathbb{R})$ and
\begin{align*}
    \widehat{U}_n:=\{f\in C(\Omega,\mathbb{R}): \mathcal{M}_h^\varphi(f)=\{\mu\}\ \mbox{for some}\ \mu\in \mathcal{Z}_n\}
    \subset U_n
\end{align*}
is dense in $C(\Omega,\mathbb{R})$.
Hence $U_n$ is an open dense subset of $C(\Omega,\mathbb{R})$. 

Then the set
\begin{align*}
    R:=\bigcap_{n\geq1} U_n
    =\left\{f\in C(\Omega,\mathbb{R}): \mathcal{M}_h^\varphi(f)\subset \bigcap_{n\geq1}\mathcal{Z}_n\right\}
\end{align*}
is a residual subset of $C(\Omega,\mathbb{R})$, and the proof is complete. 
\end{proof}


\setcounter{equation}{0}
\renewcommand{\theequation}{\Alph{section}.\arabic{equation}}

\appendix

\section{On positive entropy}\label{AppendixA}
 In this appendix, we see that for generic continuous function every relative maximizing measure has positive entropy under some constraints, which is a trivial consequence of ergodic optimization.
     \begin{proposition}
     \label{positive_entropy}
         Let $T:X\rightarrow X$ be a continuous map on a compact metric space $X$ with an ergodic invariant probability measure $\mu$ having positive entropy.
         Then there exist $\varphi\in C(X,\mathbb{R})$ and $h\in {\rm Rot}(\varphi)$ such that
         for generic $f\in C(X,\mathbb{R})$ every relative maximizing measure of $f$ with the constraint $h\in {\rm Rot}(\varphi)$ has positive entropy.
     \end{proposition}
     \begin{proof}
        By Jenkinson's theorem \cite{Jen06}[Theorem 1], there exists $\varphi \in C(X,\mathbb{R})$ such that $\mathcal{M}_{\rm max}(\varphi)=\{\mu\}$.
        Let $h:=\rv_{\varphi}(\mu)$.
        Then $\mathcal{M}^\varphi_h(f)\subset \rv_{\varphi}^{-1}(h)=\{\mu\}$ for all $f\in C(X,\mathbb{R})$.
        Since $\mathcal{M}^\varphi_h(f)$ is not empty, we have $\mathcal{M}^\varphi_h(f)=\{\mu\}$, which implies the unique relative maximizing measure of $f$ with constraint $h\in {\rm Rot}(\varphi)$ has positive entropy.
     \end{proof}

\section{Prevalent uniqueness}\label{AppendixB}
In this paper we focus on generic property of continuous functions in constrained setting. 
On the other hand, measure-theoretic ``typicality" which is known as {\it prevalence} is also important. 
A subset $\mathcal{P}$ of $C(X,\mathbb{R})$ is {\it prevalent} if there exists a  compactly supported Borel probability measure $m$ on $C(X,\mathbb{R})$ such that the set $f+\mathcal{P}$ has full $m$-measure for every $f\in C(X,\mathbb{R})$. 

Prevalent uniqueness of maximizing measures for continuous functions in the literature on (unconstrained) ergodic optimization is recently established by Morris \cite[Theorem 1]{Mor21}.
Moreover the result is generalized to more abstract statement, which yields prevalent uniqueness in constrained setting. 

\begin{corollary}
Let $T:X\rightarrow X$ be a continuous map on a compact metric space $X$.
For a continuous constraint $\varphi: X\rightarrow\mathbb{R}^d$ and a rotation vector $h\in {\rm Rot}(\varphi)$, the set
\begin{align*}
    \{f\in C(X,\mathbb{R}): \mathcal{M}^\varphi_h(f)\ \mbox{is a singleton}\}
\end{align*}
is prevalent.
\end{corollary}
\begin{proof}
   Since $\rv_{\varphi}^{-1}(h)\subset \mathcal{M}_T(X)$ is a nonempty compact set in $C(X,\mathbb{R})^*$, 
   the statement follows from Theorem 2 in \cite{Mor21}.
\end{proof}

\vspace*{33pt}

\noindent
\textbf{Acknowledgement.}~ 
The first author was partially supported by JSPS KAKENHI Grant Number 22H01138 and
the second author was partially supported by JSPS KAKENHI Grant Number 21K13816.

\noindent
\textbf{Data Availability.}~
Data sharing not applicable to this article as no datasets were generated or analyzed during the current study.


\bibliographystyle{alpha}
\bibliography{CEO_Morris.bib}

\end{document}